\documentclass[12pt]{amsart}
\usepackage{amsmath, amssymb, amsthm}
\newtheorem{lem}{Lemma}

\newtheorem{thm}{Theorem}

\begin{document}

\title{Configurations in Fractals}
\author{Nikolaos Chatzikonstantinou}

\begin{abstract}
  We define the manifold of configurations to be the quotient set of $k$ points in Euclidean space identified under congruence, and prove that compact subsets of $\mathbb{R}^d, d \geq 2$, of large Hausdorff dimension have a non-null set of configurations in them. Our method simplifies previous work in \cite{chatzikonstantinou2017rigidity} and achieves a better dimensional threshold.
\end{abstract}

\maketitle

Let $d\geq 2, k \geq 3$ and consider the Euclidean group $E(d)$ acting on $k$ points $x_1,\dots,x_k$ of $\mathbb{R}^d$ by the diagonal action
\begin{align*}
  g\cdot (x_1, \dots, x_k) = (gx_1, \dots, gx_k), \quad g \in E(d).
\end{align*}
We define the configuration space as the quotient space
\begin{align}
  \mathcal{C}^d_k & := \oplus^k\mathbb{R}^d/E(d) \label{configuration}\\
        & \cong \oplus^{k-1}\mathbb{R}^{d}/O(d). \label{equivalence}
\end{align}
(For (\ref{equivalence}) we set $x_1=0$.) The configuration space is a smooth manifold with singularities: If $k \leq d$ we have a diffeomorphism
\begin{align}
  \mathcal{C}^d_k \cong \mathcal{C}^{k-1}_k. \label{reduction}
\end{align}
Assuming $k \geq d + 1$, the set of $x\in\oplus^{k-1}\mathbb{R}^d$ with maximal rank is open and there $O(d)$ acts smoothly, freely and properly. By Theorem 21.10 of \cite{lee}, that part of $\mathcal{C}^d_k$ is a smooth manifold of dimension
\begin{align}
  \label{dimension}
  m := dk - \frac{d(d+1)}2.  
\end{align}
The residue set may be identified with $\mathcal{C}^{d-1}_k$ and together with the base case $\mathcal{C}^1_k$, which is an open cone over $\mathbb{RP}^{k-2}$, we have determined the structure of $\mathcal{C}^d_k$.

The standard Riemannian metric of $\oplus^k\mathbb{R}^d$ makes $\mathcal{C}^d_k$ into a Riemannian manifold whose distance function is given by the invariant function
\begin{align}
  d(x,y) := \min_{g\in E(d)}\|x - g \cdot y\|_{\oplus^k\mathbb{R}^d}, \quad x,y\in \oplus^k\mathbb{R}^d, \label{distancefunction}
\end{align}
see \cite{Alekseevsky}, Proposition 3.1 and its proof.

The question which is answered in this paper is the following: Given $X\subset \mathbb{R}^d$ compact, does there exist $s > 0$ such that $X^k/E(d)$ is not a null set when the Hausdorff dimension of $X$ exceeds $s$?

We prove this for
\begin{align}\label{range}
  s =
  \begin{cases}
    d - \frac{d-1}k, & k \geq d + 1, \\
    d - \frac{k-2}k, & 3 \leq k \leq d.
  \end{cases}
\end{align}
Proving the case $k \geq d + 1$ is enough, considering (\ref{reduction}), since for the case $k \leq d$ we can find a $(k-1)$-dimensional affine plane that intersects $X$ and apply the result there, see Theorem 6.6 in \cite{mattila}. The case $3 \leq k \leq d$ was first worked out in \cite{bochen}, and better results are obtained there for that range of $k$.
\begin{thm}
  Let $d\geq 2, k \geq d+1$ and $X\subset\mathbb{R}^d$ compact. If the Hausdorff dimension of $X$ exceeds $s$ (see (\ref{range})) then $X^k/E(d)$ is not a null set of $\mathcal{C}^d_k$.
\end{thm}
\begin{proof}
  Let $\mu$ be a Borel probability measure on $X$ and let $\nu$ be the probability measure on $\mathcal{C}^d_k$ defined as the pushforward of $\mu^k$ by the quotient projection of (\ref{configuration}). Let $\epsilon>0$ and define $\nu_\epsilon(x)$ for $x\in \mathcal{C}^d_k$ by
  \begin{align}
    \nu_\epsilon(x) := \epsilon^{-m}\int \chi_{[0,1]}(d(x,y)/\epsilon)d\nu(y).
  \end{align}
  For $f \in C_c(\mathcal{C}^d_k)$ and $\epsilon > 0$ small we have
  \begin{align*}
    \int f(x)d\nu_\epsilon(x) & = \int \int_{B(y, \epsilon)} \epsilon^{-m}f(x) dx d\nu(y)
  \end{align*}
  which gives, for some positive $h\in C(\mathcal{C}^d_k)$ depending on the metric of $\mathcal{C}^d_k$, the weak-* convergence $\nu_\epsilon \to h\nu$ as $\epsilon \to 0$, so
  \begin{align}
    \|\nu\|_{L^2(\mathcal{C}^d_k)} \leq C\cdot \liminf_{\epsilon\to 0}\|\nu_\epsilon\|_{L^2(\mathcal{C}^d_k)}, \quad C = \|h\|_{L^\infty(\operatorname{spt}\nu)}. \label{liminf}
  \end{align}
  We now wish to bound the right hand side of (\ref{liminf}). By the triangle inequality,
  \begin{align}
    \int_{\mathcal{C}^d_k} \nu_\epsilon(x)^2 dx & = \epsilon^{-2m}\int_{\mathcal{C}^d_k}\iint \chi_{[0,1]}(d(x,y)/\epsilon) \chi_{[0,1]}(d(x,z)/\epsilon) d\mu^k(y)d\mu^k(z)dx \nonumber \\
                                      & \leq C\epsilon^{-m} \iint \chi_{[0,1]}(d(y,z)/\epsilon) d\mu^k(y)d\mu^k(z), \quad C > 0. \label{preaction}
  \end{align}
  Using Lemma \ref{lem:epsset} together with (\ref{dimension}) we may we may bound (\ref{preaction}) by
  \begin{align}
    C\epsilon^{-d(k-1)} \int_{O(d)}\iint \prod_{1 < j \leq k}\chi\{|(z_1-z_j)-\rho(y_1-y_j)| < C'\epsilon\} d\mu^k(y)d\mu^k(z) d\rho. \label{semifinal}
  \end{align}
  We now reveal a convolution in (\ref{semifinal}). Consider the measure $\nu_\rho, \rho \in O(d)$ whose action on a measurable function $f : \mathbb{R}^d\to \mathbb{R}$ is given by
  \begin{align}
    \int f(u) d\nu_\rho(u) := \int_{\mathbb{R}^d}\int_{\mathbb{R}^d} f(z_0 - \rho y_0) d\mu(z_0) d\mu(y_0).
  \end{align}
  Then (\ref{semifinal}) may be rewritten as
  \begin{align} \label{notquite}
    C\int_{O(d)} \int\prod_{1 < j \leq k} \epsilon^{-d}\int \chi\{|u_1 - u_j| < C'\epsilon\} d\nu_\rho(u_j)d\nu_\rho(u_1)d\rho, \quad C' > 0.
  \end{align}
  As long as $\nu_\rho$ is absolutely continuous for almost every $\rho\in O(d)$, (\ref{notquite}) converges, as $\epsilon\to 0$, to a constant multiple of
  \begin{align}
    \int_{O(d)}\int_{\mathbb{R}^d} \nu_\rho^k(u)dud\rho < +\infty \label{nurho} 
  \end{align}
  which Lemma \ref{lem:nurho} shows to be finite for $s > d - \frac{d-1}k$.

  Combining (\ref{liminf}) and Lemma \ref{lem:nurho} shows that $\nu \in L^2(\mathcal{C}^d_k)$, which implies that its support is not a null set.
\end{proof}
\section{Group actions}
\begin{lem}
  \label{lem:epsset}
  Let $d \geq 2, k \geq 3$. If $z,y\in \oplus^{k}\mathbb{R}^d$ with $\|y\|\leq M$ for some $M>0$ and $d(z,y) < \epsilon$ for $\epsilon > 0$ small enough, then there exists some $C_M>0$ that depends on $M$ for which the set
  \begin{align}
    \{ \rho \in O(d) : |(z_1-z_j) - \rho(y_1 - y_j)| < C_M\epsilon, \quad 1 < j \leq k\} \label{epsset}
  \end{align}
  has Haar measure at least $\epsilon^{\frac{d(d-1)}2}$. 
\end{lem}
\begin{proof}
  Let $g_0\in E(d)$ so that $d(z,y) = \|z-g_0\cdot y\|$. Let $\rho_0\in O(d)$  be the rotational component of $g_0$. For $\rho$ close to $\rho_0$ we have
  \begin{align*}
    |(z_1-z_j) - \rho (y_1-y_j)| & \leq |(z_1-z_j) - \rho_0(y_1-y_j)| + |\rho_0(y_1-y_j) - \rho (y_1-y_j)| \\
                     & \leq 2\epsilon + 2C\|y\|\epsilon.
  \end{align*}
  This implies that an $\epsilon$-neighborhood of $\rho_0$ satisfies (\ref{epsset}), and since $\dim O(d) = {\frac {d(d-1)}2}$ the conclusion follows.
\end{proof}
The constant $C$ depends on the metric chosen for $O(d)$. The set (\ref{epsset}) may have large volume, for example it is the full group if $y=z=0$.
\begin{lem}
  \label{lem:nurho}
  Let $d \geq 2, k \geq 3$. If the Hausdorff dimension of $X$ is greater than $d - \frac{d-1}k$ then there exist some measure $\mu$ supported on $X$ for which
  \begin{align} \label{convolution}
    \int_{O(d)}\int_{\mathbb{R}^d} \nu_\rho^k(u)du d\rho < +\infty.
  \end{align}
\end{lem}
Lemma \ref{lem:nurho} is proven in \cite{bochen}, \S 3. Note that (\ref{convolution}) implies that $\nu_\rho$ is absolutely continuous for almost every $\rho$.

\section*{Acknowledgements}

I would like to thank my advisors, professor Alex Iosevich and professor Sevak Mkrtchyan for their guidance and helpful discussions during the course of my graduate studies at the University of Rochester. This work sprang from a question professor Alex Iosevich had posed to me.

\def\arraystretch{1.3}
\section*{List of Symbols}
\begin{tabular}{cp{1\textwidth}}
  $E(d)$ &  Group of Euclidean motions of $\mathbb{R}^d$. \\
  $O(d)$ & Group of rotations of $\mathbb{R}^d$. \\
  $\mathbb{RP}^{k-2}$ & Real projective space of dimension $k-2$. \\
  $\chi$ & Characteristic function of a set. \\
  $B(y,\epsilon)$ & Ball about $y\in\mathcal{C}^d_k$ of radius $\epsilon$. \\
  $C_c(\mathcal{C}^d_k)$ & Continuous functions $\mathcal{C}^d_k\to\mathbb{R}$ with compact support. \\
  $\oplus^k\mathbb{R}^d$ & Direct summand of $k$ copies of $\mathbb{R}^d$. \\
  $\|\cdot\|_{\oplus^k\mathbb{R}^d}$ & Standard Euclidean distance. \\
  $\operatorname{spt}\nu$ & Support of the measure $\nu$.
\end{tabular}

\end{document}